\documentclass[12pt]{amsart}
\usepackage{amsthm}
\usepackage{amsfonts,amssymb,amsmath}
\theoremstyle{definition}
\begin{document}
\newtheorem*{theorem}{Theorem}
\newtheorem*{lemma}{Lemma}
\title{$G_2$ representations and semistandard tableaux}
\author{William M. McGovern}
\begin{abstract}
Continuing the work of \cite{M25} we show how to realize irreducible finite-dimensional representations of the complex group $G$ of type $G_2$ via tableaux, along the way exhibiting generators of the defining ideal for the flag variety of this group.  
\end{abstract}
\maketitle

\section{Introduction}
Generalizing \cite{F97} and \cite{B86}, we showed in \cite{M25} how to construct combinatorial models of the irreducible polynomial representations of a complex classical reductive group $H$ using semistandard Young tableaux with entries in $\{0,1\ldots,n,-1\ldots,-n\}$, with $n$ a positive integer.  Using these tableaux we also constructed generators of the defining ideal of the (projective) flag variety, of $H$, all of which turn out to be homogeneous quadratic polynomials.  Here we do the same thing for the complex semisimple group $G$, of type $G_2$ and its flag variety $\mathcal F$, using entries that are weights of the standard 7-dimensional representation $V$ of $G$ rather than integers.  

\section{Roots of $G$ and the representation $V$}

Fix a maximal torus $H$ of $G$ and denote the roots in the corresponding root system as in \cite{H72}, so that the long roots are $\pm\beta,\pm (3\alpha+\beta)$, and $\pm(3\alpha+2\beta)$, while the short roots are $\pm\alpha,\pm(\alpha+\beta)$, and $\pm(2\alpha+\beta)$.  Choose a set of positive roots so that $\alpha$ is the (unique) short simple root while $\beta$ is the long simple one.  There is a subgroup $G'$ of $G$ of type $A_2$ generated by $H$ and the root subgroups corresponding to the long roots; choose a set of positive roots for it compatible with the one for $G$.  The 7-dimensional representation $V$ mentioned above has weights equal to the short roots of $G$ together with 0, each with multiplicity one; the unique highest weight is $\lambda_1=2\alpha+\beta$, the first fundamental dominant weight of $G$.  There is a symmetric bilinear form $f$ on $V$ preserved by the action of $G$ in which each weight space is paired nonsingularly with the corresponding negative weight space, whence $G$ may be viewed as a subgroup of $SO_7(\mathbb C)$. The representation $V$ decomposes under the action of $G'$ on $V$ admits three subrepresentations, one trivial, another with weights $\alpha_1:=2\alpha+\beta,-\alpha,-\alpha-\beta$, and the other with weights $\alpha_2:=\alpha+\beta,\alpha$, and $-2\alpha-\beta$ (all again with multiplicity one).  Denote the first one by $V_1$ and identify it with $\mathbb C^3$, the defining representation of $G'\cong GL_3(\mathbb C)$.  It has highest weight $\alpha_1$.  The second one, denoted $V_2$, identifies with the dual representation $(\mathbb C^3)^*$ and has highest weight $\alpha_2$; the sum of the highest weights $\alpha_1$ and $\alpha_2$ is $\lambda_2=3\alpha+2\beta$, the second fundamental dominant weight of $G$.  

We work with {\sl (vector) tableaux}, which are by definition Young diagrams with at most two rows in which each box is labelled with a weight occurring in $V$.  The weight of a tableau is defined to be the sum of its entries.  Entries are totally ordered by decreeing that $2\alpha+\beta<-\alpha<-\alpha-\beta<\alpha+\beta<\alpha<-2\alpha-\beta<0$; a tableau is {\sl semistandard} if its entries increase weakly relative to $<$ across rows and strictly down columns.  Then $G$ acts on the complex vector space $V_\lambda$ spanned by all tableaux with a fixed shape $\lambda$ via the action of $G$ on $V$ and mutilinearity. 

\section{Defining relations and the flag variety}

We need to impose relations on the tableaux in $V_\lambda$, thereby defining a quotient space on which $G$ will act irreducibly.  We start with the alternating and exchange relations of \cite{F97}, transported in the obvious way from Young to vector tableaux.  Next we impose the orthogonal relations of \cite{M25}, which state that $\sum_{i=1}^7 T_i=0$, where $T_1,\ldots,T_7$ are obtained from a diagram $T$ with all but two boxes filled by weights, by filling the first box with each of the weights $w$ of $V$ in turn and the second box by $-w$.  Next come the {\sl pairing relations} that $T_1+T_2+T_3=0$, where $T$ is a diagram with with all but two boxes in the same column filled by weights and the $T_I$ are obtained from  $T$ by filling the first box with each of the weights $w$ of $V_1$ (resp.\ $V_2$) and the second box with $-w$.   We then symmetrize these relations under the action of $G$, so that the quotient by them (and the previous relations) is $G$-stable; the resulting infinite set of relations can be replaced by a finite one, since $V_\lambda$ is finite-dimensional.    

Next come the {\sl exclusion relations} that $T=0$ whenever the tableau $T$ includes a column with entries both weights of $V_1$ or both weights of $V_2$; we again symmetrize this relation under the $G$-action.  Finally we have the {\sl transposition relations}, which decree that a tableau remains unchanged if two entries in different columns that are weights of the same $V_i$ are interchanged.  These are again symmetrized under the $G$ action.  Denote by $S_\lambda$ the quotient of $V_\lambda$ by these relations; then $G$ acts on $S_\lambda$.

Since $G$ is a subgroup of $SO_7(\mathbb C)$ it acts on the partial flag variety $\mathcal F$ consisting of all pairs $(V_1,V_2$ of $f$-isotropic subspaces of $V$ with $V_1\subset V_2$ and $\dim V_i = i$.  The $G$-suborbits of $\mathcal F$ are complete as varieties and a generic one is isomorphic to the flag variety of $G$.  The alternating, exchange, and orthogonal relations for $G_2$ tableaux show that any combination of such may be regarded as a polynomial function on $\mathcal F$; we will show that such functions satisfy the remaining relations above, so that they may be regarded as functions on the flag variety of $G$.  

To show that the coordinate functions on the flag variety satisfy the given relations, we begin by noting that the alternating, exchange, and orthogonal relations guarantee that tableaux with a single box span a representation of highest weight $2\alpha+\beta=\lambda_1$ while tableaux with a single column of two boxes span a representation of highest weight $(2\alpha+\beta)+(\alpha+\beta) =\lambda_2$.  Both of these representations appear with multiplicity one in the coordinate ring $C$ of $\mathcal F$.  The exclusion relations follow since the multiplicity of the $(\alpha+\beta)$-weight space in the representation of highest weight $\lambda_2$ is one, this weight space being spanned by the tableau having two boxes in a single column with entries $\alpha+\beta$ and 0.  (The other tableau of this shape and weight, having entries $2\alpha+\beta,$ and $-\alpha$, vanishes by the exclusion relations.)  Similarly, the transposition relation follows by looking at the multiplicity of the $(4\alpha+2\beta)$-weight space in the representation of highest weight $\lambda_1+\lambda_2 = 5\alpha+3\beta$, which is 2.   The pairing relations similarly follow by looking at the multiplicity of the 0 weight space in the representation of highest weight $\lambda_2$.

\section{Restriction to $G'$}

Following the approach of \cite{M25} for the classical groups, we will prove that $G$ acts irreducibly on $S_\lambda$ by using the branching rule for the decomposition of irreducible representations of $G$ to $G'$.  Denote the irreducible representation of $G$ with highest weight $a\lambda_1+b\lambda_2$ by $V_{a,b}$ and its counterpart for $G'$ with highest weight $a\alpha_1+b\alpha_2$ by $W_{a,b}$.  Note that the highest weight $a\lambda_1+b\lambda_2$ corresponds to tableaux of shape $(a+b,b)$ while the highest weight $a\alpha_1+b\alpha_2$ corresponds to tableaux with $a$ boxes labelled $\alpha_1$ and $b$ boxes labelled $\alpha_2$, possibly together with other boxes labelled 0.  

\begin{lemma}
The multiplicity of the representation $W_{c,d}$ in $V_{a,b}$ is\break $\min(a+2b-c-d+1,a+b-c+1,a+b-d+1,c+d-b+1,a+1,b+1)$ if $(c+d)\le a+2b,c,d\le a+b,b\le c+d$ and 0 otherwise.
\end{lemma}

\begin{proof}
We argue as in the proof of \cite[Theorem 4.1]{M90}.  The argument there coupled with the complete primality of the coordinate ring $C$ shows that whenever representations $W(c,d) $ and $W(c',d')$ appear in $V(a,b)$ and $V(a',b')$ with respective multiplicities $m$ and at least one, then $W(c+c',d+d')$ appears in $V(a+a'.b+b')$ with multiplicity at least $m$; also, whenever $W(c,c)$ appears in $V(a,b)$ with multiplicity at least two, then $W(kc,kc)$ appears in $V(ka,kb)$ with multiplicity at least $k+1$.  Using this one shows that $V(a,b)$ contains the sum of the representations $W(c,d)$ with the given multiplicities.  By counting dimensions for finitely many choices of parameters and using the polynomiality of the dimensions of $V(a,b)$ and $W(c,d)$ as functions of $(a,b)$ and $(c,d)$, respectively (with the polynomials of degree 6), one shows that $V(a,b)$ and the given sum of representations $W(c,d)$ have the same dimension, whence $V(a,b)$ coincides as a representation of $G'$ with this sum.
\end{proof}

Observe that a semistandard tableau corresponding to a highest weight vector for the action of $G'$ has all entries equal to $\alpha_1,\alpha_2$, or 0 and all $\alpha_1$s in the first row lying to the left of all $\alpha_2$s in that row.  The second row has all entries equal to $\alpha_2$ or 0, with every $\alpha_2$ lying immediately below an $\alpha_1$ in the row above and no 0 lying immediately below another one.  From this it is easy to check that the number of such semistandard tableaux $T$ with shape $(a+b,b)$ and weight $c\alpha_1+d\alpha_2$ coincides with the multiplicity of $W_(c,d)$ in $V_{a,b}$.  Note however that the $G'$-subrepresentation $R_T$ generated by $T$ does {\sl not} have partition corresponding to the shape of $T$; instead, if $T$ has weight $c\alpha_1+d\alpha_2$, then this corresponding partition is $(c+d,d,0)$.

\section{$G_2$ tableaux}
. 
We call a vector tableau $T$ a {\sl $G_2$ tableau} if it is semistandard, the weights of each $V_i$ in $T$ occur in $<$-increasing order from left to right with no two such weights in the same column, and there are no columns with entries $\alpha_1$ and $-\alpha_1$.  It is not difficult to check that $G_2$ tableaux of shape $\lambda$ provide a basis for the representation $S_\lambda$.  Then we have

\begin{theorem}
For all partitions $\lambda$ with at most two rows, the group $G$ acts irreducibly on $S_\lambda$; the resulting representation has highest weight $a\lambda_1+b\lambda_2$ if the shape of $\lambda$ is $(a+b,b)$.  Every irreducible finite-dimensional representation of $G$ arises in this way and the direct sum of all the $S_\lambda$ is isomorphic to the coordinate ring $C$ as a representation of $G$.  The relations generate the defining ideal of the flag variety $\mathcal F$.
\end{theorem}

\begin{proof}
This follows at once from the branching rule and the construction, since the tableaux $T$ defined in the last section correspond to highest weight vectors for the $G'$ action with the given highest weights.  The transposition relations guarantee that these are the only highest weight vectors for this action up to scalar multiples.
\end{proof}


\begin{thebibliography}{F97}

\bibitem[B86]{B86} A.\ Berele, \textsl{Construction of Sp modules by tableaux}, Linear and Multilinear Algebra~{\bf19} (1986), 299--307.

\bibitem[F97]{F97} W.\ Fulton, \textsl{Young tableaux: with applications to representation theory and geometry}, London Math.\ Soc.\ Student Texts~{\bf35}, Cambridge University Press, New York, 1997.

\bibitem[H72]{H72} J.\ E.\ Humphreys, \textsl{Introduction to Lie Algebras and Representation Theory}, Springer-Verlag, Berlin-Heidelberg-New York, 1972.

\bibitem[M90]{M90} W.\ M.\ McGovern, \textsl{A branching law for Spin$(7,\mathbb C)$ and its applications to unipotent representations}, J.\ Alg.~{\bf130} (1990), 166--175.

\bibitem[M25]{M25} W.\ M.\ McGovern, \textsl{Symplectic and orthogonal tableaux revisited}, arxiv:2504.14030, 2025.

\end{thebibliography}
\end{document}